\documentclass[letterpaper, 10 pt, conference]{ieeeconf}  

\makeatletter
\let\NAT@parse\undefined
\makeatother

\IEEEoverridecommandlockouts                              

\usepackage{graphics} 
\usepackage{amsmath} 
\usepackage{amssymb}  

\usepackage{amssymb}
\usepackage{tikz-cd} 
\usepackage{float}
\usepackage[numbers]{natbib}        

\newtheorem{theorem}{Theorem}[section]
\newtheorem{proposition}[theorem]{Proposition}
\newtheorem{lemma}[theorem]{Lemma}
\newtheorem{corollary}[theorem]{Corollary}

\newtheorem{definition}[theorem]{Definition}

\newcommand{\tr}{\textnormal{tr}}

\DeclareMathOperator{\im}{im}

\makeatletter
\DeclareRobustCommand{\der}[1]{%
  \@ifnextchar\bgroup{\@der{#1}}{\@der{}{#1}}}
\newcommand{\@der}[2]{\frac{d#1}{d#2}}
\makeatother
\makeatletter
\DeclareRobustCommand{\dder}[1]{%
  \@ifnextchar\bgroup{\@dder{#1}}{\@dder{}{#1}}}
\newcommand{\@dder}[2]{\frac{d^2#1}{d#2^2}}
\makeatother

\makeatletter
\DeclareRobustCommand{\pder}[1]{%
  \@ifnextchar\bgroup{\@pder{#1}}{\@pder{}{#1}}}
\newcommand{\@pder}[2]{\frac{\partial#1}{\partial#2}}
\makeatother

\makeatletter
\renewcommand*\env@matrix[1][*\c@MaxMatrixCols c]{%
  \hskip -\arraycolsep
  \let\@ifnextchar\new@ifnextchar
  \array{#1}}
\makeatother

\title{\LARGE \bf
Model Reduction of Semistable Distributed Parameter Systems
}

\author{Ingvar Ziemann and Yishao Zhou
\thanks{Ingvar Ziemann is with the School of Electrical Engineering and Computer Science, KTH Royal Institute of Technology, SE-100 44
Stockholm, Sweden
        {\tt\small ziemann@kth.se}}%
\thanks{Yishao Zhou is with the Department of Mathematics, Stockholm University, SE-106 91 Stockholm, Sweden
        {\tt\small yishao@math.su.se}}%
}

\begin{document}

\maketitle
\thispagestyle{empty}
\pagestyle{empty}

\begin{abstract}
The model reduction problem for semistable infinite-dimensional control systems is studied in this paper. In relation to these systems, we study an object we call the semistability Gramian, which serves as a generalization of the ordinary controllability Gramian valid for semistable systems. This Gramian is then given geometric as well as algebraic characterization via a Lyapunov equation. We then proceed to show that under a commutativity assumption relating the original and reduced systems, and as long as the semistability is preserved, we may derive a priori error formulas in $\mathcal{H}_2$-norm in terms of the trace of this Gramian.
\end{abstract}

\section{Introduction}

In this article we investigate the model reduction problem for distributed parameter systems. In particular we are interested in deriving and characterizing an a priori formula for the $\mathcal{H}_2$-error between the original and reduced models for such systems. Here, we restrict ourselves to systems which are exponentially semistable; that is, systems for which the dynamics are guaranteed to eventually converge, but to where precisely is allowed to depend on the initial conditions. A large motivation for the study of semistability is system thermodynamics, which naturally exhibit semistability for certain boundary conditions. However, partial differential systems such as these suffer from infinite-dimensionality which makes them computationally intense. As such it is important to find approximating systems which are close in norm and behave similarly. Moreover, semistability is of increasing practical interest as the importance of networked systems, which are often semistable, continues to grow. 

Regarding semistability, recent advances have been made in the context of networked systems by for instance \cite{hui2011optimal} and even specifically in model reduction by \cite{besselink2016clustering}. In \cite{cheng2017reduction},  the authors introduce the idea of an augmented Gramian. This concept turns out to be central for us too and many of our theorems are generalizations of theirs, carried over from the setting of network systems to the more general case of semistable distributed parameter systems. 

However, not much recent work in distributed parameter systems has been completed in the context of semistability, except perhaps \cite{hui2013semistability}. Knowledge of model reduction for distributed parameter systems is also fairly sparse, with most work focused on Hankel norm approximations as in \cite{sasane2002hankel}, or directly on numerical schemes, as in \cite{atwell2004reduced}. To this end, \cite{curtain2003model} states that our numerical capabilities far outweigh our theoretical understanding of these approximations. In contrast, the finite-dimensional theory of model reduction puts much emphasis on the $\mathcal{H}_2$-norm. Our contribution here is to extend known results concerning the $\mathcal{H}_2$-norm problem to the infinite-dimensional setting.

\section{Problem Setting}
By model reduction, we mean that given a system:
\begin{align}
\begin{cases}
\dot x = Ax+ Bu\\
y=Cx, x(0)=x_0\\
\end{cases}
\tag{$\Sigma$}
\end{align}
to find a reduced system
\begin{align}
\begin{cases}
\dot v = \hat Av+ \hat Bu\\
\hat y=\hat C v, v(0)=v_0\\
\end{cases}
\tag{$\hat \Sigma$}
\end{align}
which approximates the initial system well both qualitatively and quantitatively. This paper considers a class of model reductions which, roughly speaking, arise when one projects the dynamics onto a sub-collection of eigenvectors of $A$. We shall later see that this automatically guarantees the preservation of the important system-theoretic properties of (semi)stability and approximate controllability. Most importantly, we will be able to give an exact a priori error formula for these in $\mathcal{H}_2$-norm in Theorem \ref{mytheorem}.

We do not make any assumptions about the dimensionality of the reduced model and it is interesting to note the possibility for the reduced model to still be infinite-dimensional. This is for instance the case when one starts with a PDE, say the heat equation, on some high-dimensional manifold $M$ and then reduces the number of equations, resulting in a PDE on a lower-dimensional manifold $N$, $\dim N \ll \dim M$. Note that in this example the state space of both the original and reduced systems are typically function spaces such as $L^2(M), L^2(N)$ or corresponding Sobolev spaces and thus infinite-dimensional. As for reductions that result in finite-dimensional models, examples include projections onto finite subcollections of eigenspaces of the $A$-operator.

To make matters precise, assume that $\Sigma(A,B,C)$ and $\hat \Sigma(\hat A , \hat B, \hat C)$ are such that $A, \hat A$ generate $C_0$-semigroups, $S(t)$ and $\hat S(t)$, on separable Hilbert spaces $X$, $V\subset X$ and $B, \hat B$, $C,\hat C$ are bounded linear operators. We denote the inner product on $X$ as well as $V$ by $\langle \cdot,  \cdot \rangle$ and the associated norm by $\| \cdot \|$. The inputs and outputs $u,y$ are assumed to lie in, possibly infinite-dimensional, Hilbert spaces $U$ and $Y$. In the finite-dimensional case this corresponds to finding matrices $(\hat A , \hat B, \hat C)$ which are of lower rank than $(A,B,C)$.

If $A$ is bounded the associated semigroup takes the form $S(t)=e^{At}=\sum_{i=0}^\infty \frac{A^i t^i}{i!}$. We consider the more general situation where $A$ (and $\hat A$ analogously) is only defined on a subspace $D(A)\subset X$. It will also often be necessary to talk about the adjoint of these operators. The semigroup $S^*(t)$ generated by $A^*$ actually coincides with $[S(t)]^*$, the adjoint of the semigroup generated by $A$. That is, the operation of taking adjoints commutes with that of taking semigroups.

Let $H\in \mathfrak{B}(X)$ be the space of bounded linear operators on $X$. Just as the norm on $X$, we also denote the supremum norm on $\mathfrak{B}(X)$ by $\| \cdot \|$ and it should be clear from context which is used. For $H \in \mathfrak{B}(X)$, we define its trace by $\tr H = \sum_{i=1}^\infty \langle H e_i, e_i \rangle$ where $(e_i)$ is any orthogonal basis for $X$. If $\tr HH^*$ is finite, we say that $H$ is Hilbert-Schmidt.

The impulse response of  a system $\Sigma(A,B,C)$ is given by $h(t)=CS(t)B$ for all $t\geq 0$. This allows us to define the $\mathcal{H}_2$-norm as 
$\|\Sigma \|_{\mathcal{H}_2}=\sqrt{\int_0^\infty \tr\big(h(t)h^*(t)\big)dt}$. As for integrals, they are to be interpreted depending on the context; integrals of functions are Bochner integrals and integrals of linear operators are Pettis integrals. See \cite{diestel1977vector} for details.

We also recall the following: The reachability space, $\mathcal{R}$, of $\Sigma(A,B,-)$ is given by set of all states that can be attained by some control from the origin. If $\mathcal{R}$ is dense in $X$, we say that $\Sigma(A,B,-)$ is approximately controllable.  For brevity of exposition we will focus exclusively on approximate controllability, but most results carry over to approximate observability by adjusting definitions appropiately and duality. A more detailed discussion of all these definitions and concepts can be found in \cite{curtain2012introduction}.

\section{Semistability}
Now, we make precise the notion of stability studied here.
\begin{definition}

\label{exponentialsemistability}
Suppose $A$ generates a $C_0$-semigroup $S(t)$ on $X$. $A$, $S(t)$ are said to be exponentially semistable if for every $x\in X$ there exists $x_e \in \ker A$  and scalars $M,\mu>0$ such that $\|S(t)x-x_e\|\leq Me^{-\mu t} \| x-x_e\|$. 
\end{definition}
Note that every member of $\ker A$ is an equilibrium point of the dynamical system given by $S(t)$. That is, we have that $S(t) \ker A = \ker A$. Further, since $S(t)x\to x_e$ strongly as $t\to\infty$, it makes sense to call $x_e\in \ker A$ the equilibrium point corresponding to $x\in X$ and standard arguments show that any such equilibrium is Lyapunov stable. To familiarize us with the definition, we also note that in the finite-dimensional case, $A \in \mathfrak{B}(\mathbb{C}^n)$, \cite{bernstein1995lyapunov} shows that $A$ is exponentially semistable if and only if $\Re \lambda \leq 0$ for all eigenvalues $\lambda$ of $A$ and all eigenvalues with $0$ real part are semisimple and have no imaginary part.

Motivated by the existence of the strong limit for each $x$ of $S(t)x$, we define $S_\infty = \lim_{t\to \infty} S(t)$. This mapping takes initial conditions to corresponding equilibrium points.

\begin{lemma}
\label{sinftyconverges}
If $S(t)$ is an exponentially semistable semigroup the limiting operator $S_\infty: X \to \ker A\subset X$ of $S(t), t \to \infty$ exists, is bounded and idempotent.
\end{lemma}

\begin{proof}
We begin by estimating the norm:
\begin{align*}
\|S(t)-S(s)\|&=\sup_{\|x\|=1,x\in X} \|S(t)x-S(s)x\|\\
&=\sup_{\|x\|=1,x\in X} \|S(t)x-x_e-(S(s)x-x_e)\|\\
&\leq \sup_{\|x\|=1,x\in X} 2Me^{-\mu \min(s,t)} \| x-x_e\|\\
&\leq \sup_{\|x\|=1,x\in X} 2Me^{-\mu \min(s,t)}(1+\|x_e\|).
\end{align*}
Note that this still depends on the distance from of the origin of the equilibrium point $\|x_e\|$. To alleviate this, we will establish a uniform bound on the family $S(t)$. Observe that by assumption of semistability, for each $x\in X$
\begin{align*}
\|S(t)x\| \leq \|x_0\|+\|S(t)x-x_0\| \leq \|x_0\|+ M \|x -x_0\|
\end{align*}
so that $\sup_{t}\|S(t)x\|< \infty$ for each $x\in X$. By the Banach-Steinhaus Theorem this means that $\|S(t)\|$ is uniformly bounded, by say $M'$. Suppose now that there exists $x$ with $\|x_e\|>M'$. If we estimate $\|x_e\|$ we find that
\begin{align*}
\|x_e\|&=\lim_{t\to \infty} \|S(t)x\| \leq \lim_{t\to \infty} \|S(t)\|\|x\|\\
&=\lim_{t\to \infty} \|S(t)\| \leq M'.
\end{align*}
contradicting $\|x_e\|>M'$. Hence
\begin{align*}
\|S(t)-S(s)\|&\leq 2Me^{-\mu \min(s,t)}(1+M')
\end{align*}
and so since $S(t)\in \mathfrak{B}(X)$ is Cauchy in $t$, there exists a limiting operator $S_\infty$ which is bounded by completeness of $\mathfrak{B}(X)$. Moreover, $S_\infty x = x_e \in \ker A$ and moreover
\begin{align*}
0&=\|x_e-x_e\| = \|S(t)x_e-x_e\| = \lim_{t\to \infty} \|S(t)x_e-x_e\|\\
&=\|S_\infty x_e-x_e\|
\end{align*}
so that $S_\infty x_e =x_e$. That is, $\forall x\in X$, $S_\infty^2 x = S_\infty x$.
\end{proof}
 The lemma above emphasizes the importance of the generator kernel, the proof of which shows us that $S(t)-S_\infty$ has nice stability properties. The operator $S_\infty$ has a particularly nice interpretation when $A$ is self-adjoint.

\begin{corollary}
Suppose that $A$ is self-adjoint. Then $S_\infty$ is the orthogonal projection onto the kernel of $A$.
\end{corollary}

We now proceed to characterize semistability via the operator $S_\infty$.

\begin{theorem}
\label{semichars}
If $S(t)$ is a $C_0$-semigroup with generator $A$ the following are equivalent:
\begin{enumerate}
\item $S(t)$ is exponentially semistable.
\item There exists a bounded operator $S_\infty :X \to \ker A$ which is idempotent on $\ker A$ and constants $\mu,L>0$ such that for every $x\in X$  $\|(S(t)-S_\infty)x\|\leq  Le^{-\mu t}\|x\|.$
\end{enumerate}
\end{theorem}

\begin{proof}
1. implies 2. by Lemma \ref{sinftyconverges} and since
\begin{align*}
\|(S(t)-S_\infty)x\|&=\|S(t)x-x_e\|\leq M e^{-\mu t} \|x-x_e\|\\
&= M e^{-\mu t} \|x-S_\infty x\|\\
& \leq \|I-S_\infty\| M e^{-\mu t} \|x\|.
\end{align*}
so $S_\infty$ is the desired operator. Conversely, it is easy to see that 2. implies 1. since one may write
\begin{align*}
\|(S(t)x-S_\infty x)\|&=\|(S(t)-S_\infty)(x- S_\infty x)\|\\&\leq L e^{-\mu t} \| x-S_\infty x\|
\end{align*}
so $S_\infty x$ is the equilibrium point corresponding to $x$.
\end{proof}

The theorem makes precise that our equilibria depend on the initial condition in the sense that the dynamics governed by $S(t)-S_\infty$ possesses a unique equilibrium.

\section{The Gramian}

The ordinary controllability Gramian is not suitable for our analysis, since $A$ having a nontrivial kernel results in an ill-defined integral. To alleviate this, we use a trick first employed in \cite{cheng2017reduction}, but adjusted to our more general situation.

\begin{definition}
\label{semigram}
The semistability Gramian of an exponentially semistable system $\Sigma(A,B, -)$  is given by
\begin{align*}
P_\infty=\int_0^\infty (S(t)-S_\infty)BB^*(S(t)-S_\infty)^* dt
\end{align*}
where the integral is taken in the sense of Pettis, see \cite{diestel1977vector}.
\end{definition}

If we had not adjusted by $S_\infty$ in the definition above, the integral would not converge. We now show that this adjustment assures convergence.

\begin{lemma}
The semistability Gramian of an exponentially semistable system $\Sigma(A,B,-)$ exists and is bounded; $P_\infty\in \mathfrak{B}(X)$.
\end{lemma}
\begin{proof}
Define a family of operators
\begin{align*}
P_t=\int_0^t (S(s)-S_\infty)BB^*(S(s)-S_\infty)^* ds.
\end{align*}
Using Theorem \ref{semichars} to bound $(S(t)-S_\infty)$ by an exponential growth condition, we obtain pointwise
\begin{align*}
\| P_tx\|&=\left\|\int_0^t (S(s)-S_\infty)BB^*(S(s)-S_\infty)^*x ds  \right\|   \\
&\leq \int_0^t \| (S(s)-S_\infty)BB^*(S(s)-S_\infty)^* x\| ds \\
&\leq \|B\|^2  L^2 \int_0^te^{-2\mu s}  ds\|x\| =  \|B\|^2L^2 \frac{1-e^{-2\mu t}}{2\mu}\|x\|\\
&\leq \frac{\|B\|^2 L^2}{2\mu}\|x\|.
\end{align*}
Observe that the constants $L, \mu$ a priori depend on $x$. More precisely there exists a pointwise norm bound for $P_t x, x\in X$ which however is independent of $t$. 
\begin{align*}
\|P_\infty\| = \| \lim_{t \to \infty} P_t \| = \lim_{t \to \infty} \| P_t \| \leq \lim_{t\to \infty} K =K
\end{align*}
by Banach-Steinhaus, which gives $\|P_t \| \leq K$ uniformly.
\end{proof}

Now, we are ready to extend the classic Lyapunov equation result to the semistability Gramian.

\begin{theorem}
\label{semiLyapunov}
For every $x\in D(A^*)$, $P_\infty$ satisfies the semistability  Lyapunov equation
\begin{align*}
AP_\infty x+P_\infty A^*x=-(I-S_\infty)BB^*(I-S_\infty)^*x.
\end{align*}
\end{theorem}

\begin{proof}
Let $x,x' \in D(A^*)$ and observe that, if integrable, we have formally
\begin{align*}
&\int_0^\infty \der{t} \langle B^* [S(t)-S_\infty]^*x,B^*[S(t)-S_\infty]^*x'\rangle dt\\
&=-\langle B^*(I-S_\infty)^*x,B^*(I-S_\infty)^*x' \rangle.
\end{align*}
Moreover, using the fact that $\der{S(t)}{t} = A S(t)=S(t)A$,
\begin{align*}
& \der{t} \langle B^* [S(t)-S_\infty]^*x,B^*[S(t)-S_\infty]^*x'\rangle\\
 &=\langle B^* A^*[S(t)]^*x,B^*[S(t)-S_\infty]^*x'\rangle\\&+\langle B^* [S(t)-S_\infty]^*x,B^*A^*[S(t)]^*x'\rangle.
\end{align*}
Now
\begin{align*}
&\int_0^\infty \langle B^* A^*[S(t)]^*x,B^*[S(t)-S_\infty]^*x'\rangle dt\\ &= \int_0^\infty \langle [S(t) A]^*x,BB^*[S(t)-S_\infty]^*x'\rangle dt\\
&= \int_0^\infty \langle [S(t)-S_\infty] A]^*x,BB^*[S(t)-S_\infty]^*x'\rangle dt\\
&= \int_0^\infty \langle A^* x, [S(t)-S_\infty] BB^*[S(t)-S_\infty]^*x'\rangle dt\\
&=  \left\langle A^* x,\int_0^\infty [S(t)-S_\infty] BB^*[S(t)-S_\infty]^*x'dt\right\rangle \\
&=\langle A^*x, P_\infty x' \rangle
\end{align*}
where we used that Lemma \ref{sinftyconverges} implies that $S_\infty A =0$. Similar computations show that
\begin{align*}
\langle B^* [S(t)-S_\infty]^*x,B^*A^*[S(t)]^*x'\rangle=\langle P_\infty x, A^*x' \rangle.
\end{align*}
Therefore
\begin{align*}
&\langle P_\infty x, A^*x' \rangle+\langle A^*x, P_\infty x' \rangle\\
&=-\langle B^*(I-S_\infty)^*x,B^*(I-S_\infty)^*x' \rangle.
\end{align*}
Since $D(A^*)$ is dense in $X$ this implies
\begin{align*}
AP_\infty x+P_\infty A^*x=-(I-S_\infty)BB^*(I-S_\infty)^*x
\end{align*}
for every $x\in D(A^*)$. To finish the proof, note that the required integrability to justify our formal computations follows from
\begin{align*}
&\left| \der{t} \langle B^* [S(t)-S_\infty]^*x,[S(t)-S_\infty]^*x'\rangle \right|\\ &\leq \left | \langle B^* A^*[S(t)]^*x,B^*[S(t)-S_\infty]^*x'\rangle \right |\\
&+\left |\langle B^* [S(t)-S_\infty]^*x,B^*A^*[S(t)]^*x'\rangle \right |\\
&\leq \|A^*x\| \|x'\| \|B^*\|^2 L^2 e^{-2\mu t}\\
&+\|A^*x'\| \|x\| \|B^*\|^2 L^2 e^{-2\mu t}
\end{align*}
where we used our characterization of semistability in Theorem \ref{semichars} to obtain a uniform bound, pointwise in time on $S(t)-S_\infty$.
\end{proof}

Unfortunately, the semistability of the system, the fact that $A$ might have a kernel, implies the possibility for non-uniqueness of the Lyapunov equation above. The following two results try to specify exactly which solution of the Lyapunov equation we are interested in. 

\begin{lemma}
\label{lyupchars}
Assume that $P_1$ is a self-adjoint solution of the semistability Lyapunov equation
\begin{align*}
&\langle P_1x, A^*x' \rangle+\langle A^*x, P_1 x' \rangle\\&=-\langle B^*(I-S_\infty)^*x,B^*(I-S_\infty)^*x' \rangle
\end{align*}
where $A$ is the infinitesimal generator for an exponentially semistable $C_0$-semigroup on a separable Hilbert space, $X$, and suppose $x,x' \in D(A^*)$. If $P_2$ is another self-adjoint operator, then each of the statements below implies the next. If $A$ in addition is self-adjoint, all the statements are equivalent.
\begin{enumerate}
\item $P_2$ satisfies the semistability Lyapunov equation.
\item $\Delta= P_2-P_1$ satisfies for each $x,x' \in D(A^*)$
\begin{align*}
\langle S_\infty x, \Delta S_\infty x' \rangle = \langle x , \Delta x' \rangle.
\end{align*}
\item There exists an operator $\Pi :X \to X$ that  $\Pi$ maps onto a subspace $W$ of $\ker A^*$  such that the solutions satisfy the relation $P_2=P_1+\Pi$.
\end{enumerate}
\end{lemma}
\begin{proof} We first show that $1 \Rightarrow 2$.
Let $P_2$ be another self-adjoint solution of the Lyapunov equation and consider $\Delta = P_1-P_2$. For $x,x' \in D(A^*)$. It is not hard to see that
\begin{align*}
\langle x, \Delta A x'\rangle + \langle A x, \Delta x' \rangle =0.
\end{align*}
If we let $x= S(t)x_0, x'=S(t)x'$, this can be rewritten as
\begin{align*}
0&=\langle S(t) x_0, \Delta A S(t)x_0'\rangle + \langle A S(t)x_0, \Delta S(t)x_0' \rangle\\
&=\langle S(t) x_0, \Delta \der{t} S(t)x_0'\rangle + \langle \der{t} S(t)x_0, \Delta S(t)x_0' \rangle\\
&=\der{t} \langle S(t)x_0, \Delta S(t)x_0'\rangle.
\end{align*}
Integrating this equation from $0$ to $\infty$ we obtain
\begin{align*}
\langle S_\infty x_0, \Delta S_\infty x_0' \rangle = \langle x_0 , \Delta x_0' \rangle.
\end{align*}

Now $2 \Rightarrow 3$. To see this, we take $\Pi=\Delta$, since 
\begin{align*}
&\langle S_\infty x, \Delta S_\infty x' \rangle = \langle x , \Delta x' \rangle\\
\Leftrightarrow & \langle S_\infty^* \Delta S_\infty x,  x' \rangle = \langle \Delta x ,  x' \rangle
\end{align*}
and since $x,x' \in D(A^*)$ where $D(A^*)$ is dense in $X$, we indeed have for any $\bar x \in X$
\begin{align*}
\Delta \bar x= S_\infty^* \Delta S_\infty \bar x =S_\infty^* (\Delta S_\infty \bar x) \in \ker A^*.
\end{align*}

Finally, $3\Rightarrow 1$ in the case $A$ is self-adjoint. This follows since by construction of $\Pi$ we have geometrically that $\Pi$ maps to the kernel of $A^*$, so $A\Pi= A^*\Pi=0$ since $A$ is self-adjoint. But then also $0=(A\Pi)^*=\Pi^*A^*$. Thus
\begin{align*}
&\langle P_2x, A^*x' \rangle+\langle A^*x, P_2 x' \rangle\\&=\langle (P_1+  \Pi) x, A^*x' \rangle+\langle A^*x, (P_1+ \Pi) x' \rangle\\
&=\langle I x,(P_1+  \Pi)^* A^*x' \rangle+\langle (P_1+  \Pi)^* A^*x, Ix' \rangle\\
&=\langle I x,(P_1)^* A^*x' \rangle+\langle (P_1)^* A^*x, Ix' \rangle\\
&=\langle P_1x, A^*x' \rangle+\langle A^*x, P_1 x' \rangle\\
&=-\langle B^*(I-S_\infty)^*x,B^*(I-S_\infty)^*x' \rangle
\end{align*}
by direct computation.
\end{proof}

Using this lemma, we can explicitly compute the semistability Gramian without reference to the semigroup whenever the generator is self-adjoint and thus $S_\infty =\pi_{\ker A}$ by Corollary 3.3. We show this below.

\begin{theorem}
\label{theexplicittheorem}
Let $A$ be the self-adjoint exponentially semistable generator of a $C_0$-semigroup $S(t)$ on a separable Hilbert space, $X$, and let $B$ be bounded. Suppose further that $P$ is an arbitrary solution to the semistability Lyapunov equation
\begin{align*}
&\langle Px, A^*x' \rangle+\langle A^*x, P x' \rangle\\
&=-\langle B^*(I-S_\infty)^*x,B^*(I-S_\infty)^*x' \rangle
\end{align*}
then the semistability Gramian can be computed as
\begin{align*}
P_\infty= P-S_\infty P.
\end{align*}
In particular, $P_\infty$ is the unique solution to the semistability Lyapunov equation satisfying the constraint
\begin{align*}
P_\infty= (I-S_\infty)P_\infty.
\end{align*}
\end{theorem}
\begin{proof}
Observe that
\begin{align*}
S_\infty S(t)= \lim_{s\to \infty} S(s) S(t)= \lim_{s\to \infty} S(t+s)=S_\infty
\end{align*}
and by Lemma \ref{sinftyconverges} we already have $S_\infty^2=S_\infty$. This implies that
\begin{align*}
S_\infty [S(t)-S_\infty]=0.
\end{align*}
which in turn implies that
\begin{align*}
S_\infty P_\infty=0.
\end{align*}
If $P$ is any other solution to the semistability Lyapunov equation, substituting the third characterization of Lemma \ref{lyupchars} yields
\begin{align*}
S_\infty (P+\Pi)&=0, \textnormal { or}\\
S_\infty \Pi &= -S_\infty P, \textnormal{ so that} 
\Pi = -S_\infty P.
\end{align*}
In the final step we used that $S_\infty$ acts identically and is  idempotent on $\im \Pi\subseteq\ker A^*= \ker A$.
\end{proof}

\section{Model Reduction and Error Estimates}

As mentioned before, the ultimate aim of our study of the semistability Gramian is to derive an error formula for model reduction. Our application here is to a class of model reductions which roughly speaking correspond to mode truncation of the generator $A$.

\begin{definition}
\label{invariantmodelreductiondefinition}
An invariant model reduction of $\Sigma$ onto $V$ is a triple $(\pi,\sigma , \hat A)$ where $\pi : X \to V$ is a bounded surjective operator, $\sigma : V \to X$ is a bounded operator and $\hat A$ satisfies $\hat A \pi x = \pi A x$ for all $x \in D(A)$. The reduced input and output operators are given by $\hat B = \pi B$ and $\hat C=C\sigma$.
\end{definition}

The full power of the commutativity assumption is brought to life by the following theorem, found originally in \cite{atay2017lumpability} for the infinite-dimensional case. 

\begin{theorem}[\cite{atay2017lumpability}]
\label{unoriginaltheorem}
Suppose that $\pi : X \to V$ is a bounded linear map and that $A$ is the infinitesimal generator of a $C_0$-semigroup $S(t)$. Then the following are equivalent:
\begin{enumerate}
\item $\ker \pi$ is $S(t)$-invariant for each $t\geq 0$.
\item There exists $\hat A: \pi (D(A)) \to V$ generating a $C_0$-semigroup $\hat S(t)$ on $V$ with $\pi A= \hat A \pi$ on $D(A)$ and in this case $\pi S(t) = \hat S(t) \pi$ for each $t\geq 0$ on $X$.
\end{enumerate}
\end{theorem}

To guide our intuition, note that when $A$ admits an orthogonal eigenvalue-eigenvector decomposition, the (closed) $S(t)$-invariant subspaces of $X$ are linear combinations of its eigenvectors. The precise statement can be found as Lemma 2.5.8 in \cite{curtain2012introduction} which shows that this reasoning remains valid for the class of Riesz spectral operators, a class of unbounded operators admitting an SVD-like decomposition.

The first application of the above theorem to model reduction is that the original system's stability properties are preserved under the class of reductions considered here.

\begin{proposition}
If $A$ is semistable on the Hilbert space $X$ and $(\pi,\hat A)$ is an invariant linear model reduction onto $V$ then $\hat A$ is semistable on $V$. If $A$ in addition is stable, then so is $\hat A$.
\end{proposition}

\begin{proof}
Let $v \in V$. First, observe that any $v\in V$ can  be written $v=\pi x = \pi_| x$ for $x=\pi_|^{-1}v \in X$ since the map $\pi$ restricts to a bounded linear operator with bounded inverse $\pi_|:=\pi_{|(\ker \pi)^\perp}: (\ker \pi)^\perp \to V$ which we obtain by an application of the Open Mapping Theorem. Denote the equilibrium point of $x$ by $x_e$, which exists by semistability of $S$. Then using the second characterization of semistability in Theorem \ref{semichars}
\begin{align*}
\|\hat S(t) v-\pi x_e\| &=\| \hat S(t) \pi x - \pi x_e\|=\|\pi S(t) x- \pi x_e \| \\
&\leq \|\pi\| \|S(t)x-\pi x_e\|\leq |\pi\| L e^{-\mu t } \| x\|\\
& =  |\pi\| L e^{-\mu t } \| \pi_|^{-1}v\| \leq  \frac{\|\pi\|}{\|\pi_|^{-1}\|} L e^{-\mu t } \| v\| .
\end{align*}
The desired equilibrium point is thus given by $\pi x_e$.

Now, if $A$ is stable, the only equilibrium point is $x_e=0$ and so the bound above reduces to exponential stability.
\end{proof}

We will now see that also approximate controllability is preserved.

\begin{proposition}
Let $\Sigma(A,B,-)$ be an approximately controllable control system on the Hilbert space $X$ and $(\pi,\hat A)$ be an invariant linear model reduction onto $V$. Then the reduced model $\Sigma(\hat A,\hat B,-)$ is approximately controllable on the reduced space $V$.

\end{proposition}

\begin{proof}
Suppose that the reachability subspace of $\Sigma(A,B,-)$ is dense in $X$. Any $x\in X$ can thus be written
\begin{align*}
x=\lim_{n\to \infty} \int_0^{\tau_n} S(\tau-s) B u_n ds
\end{align*}
for $\tau_n>0, u_n \in U$. But for any $v\in V$, the model reduction satisfies for some $x\in X$
\begin{align*}
v=\pi x& = \pi \lim_{n\to \infty}\int_0^{\tau_n} S(\tau-s) B u_n ds\\
&= \lim_{n\to \infty}\int_0^{\tau_n} \pi S(\tau-s) B u_n ds\\
&=  \lim_{n\to \infty}\int_0^{\tau_n} \hat S(\tau-s)\pi B u_n ds\\
&=\lim_{n\to \infty}\int_0^{\tau_n} \hat S(\tau-s)\hat B u_n ds.
\end{align*}
We conclude: for every $v\in V$ there is a sequence of elements in the reachability subspace of $\Sigma(\hat A, \hat B, -)$ that converge to $v$, i.e., the reachability subspace for the reduced model is also dense. The interchanges of the limit and integral with $\pi$ are justified by that first, $\pi$ is bounded, and second by that the integrands are bounded operators.
\end{proof}

Although we have focused on approximate controllability in this exposition, it is not hard to see that analogous results can be derived for approximate observability by the natural duality.

As a first step toward our norm guarantees, the following result gives  trajectory-wise proximity of the reduced system to the original system.

\begin{proposition}
Let $\Sigma(A,-,-)$ be an exponentially \\ semistable system and suppose that $(\pi,\sigma,\hat A)$ is an invariant model reduction of this system with $\sigma \pi$ restricting to the identity on $\ker A$. Then for all initial conditions, $x\in X$ we have that $\|S(t)x-\sigma \hat S (t)\pi x\| \to 0$.

\end{proposition}

\begin{proof}
By commutativity $\hat S(t) \pi x = \pi S(t) x$. So we may write for any $x\in X$ with equilibrium point $x_e\in \ker A$
\begin{align*}
\|S(t)x-\sigma \pi S(t)x \|&=\|(S(t)x-x_e)-(\sigma\pi S(t) x-x_e)\| \\
&=\|(S(t)x-x_e)-(\sigma\pi S(t) x-\sigma\pi x_e)\| \\
&\leq \|I-\sigma \pi \| M e^{-\mu t} \|x-x_e\|
\end{align*}
proving the result.
\end{proof}

The synchronization result above guides our intuition for the hypotheses necessary for the main result on the $\mathcal{H}_2$-norm error, which we state immediately below.

\begin{theorem} 
\label{mytheorem}
Suppose that $\Sigma(A,B,C)$ is a distributed parameter system on a separable Hilbert space $X$ where  $A$ generates a semistable $C_0$-semigroup $S(t)$ that $B$ and $C$ are bounded and that $(\sigma, \pi, \hat A)$ is an invariant model reduction thereof where $\sigma \pi$ restricts to the identity on $\ker A$. Then if $(I-\sigma\pi)S(t)$ is Hilbert-Schmidt the model error is given by
\begin{align*}
\|\Sigma-\hat \Sigma\|_{\mathcal{H}_2}=\sqrt{\tr\Big(C(I-\sigma\pi) P_\infty(I-\sigma \pi)^*C^* \Big)}
\end{align*}
where $P_0$ is the semistability Gramian of $\Sigma$ which for $x\in D(A^*)$ satisfies
\begin{align*}
APx+PA^*x=-(I-S_\infty)BB^*(I-S_\infty)^*x.
\end{align*}
\end{theorem}
\begin{proof}
Since $\pi$ is an invariant model reduction
\begin{align*}
 h(t)-\hat h(t) &=CS(t)B-C\sigma\hat S(t) \hat B\\
 &=C(I-\sigma \pi) S(t) B.
\end{align*}
As the right hand side above is a composition of bounded and Hilbert-Schmidt operators, $h(t)-\hat h(t)$ is Hilbert-Schmidt. Now
\begin{align*}
&\|\Sigma-\hat \Sigma\|^2_{\mathcal{H}_2}\\
&=\int_0^\infty  \tr\Big([h(t)-\hat h(t)][h(t)-\hat h(t)]^*\Big)dt  \\
&=\int_0^\infty\sum_{i=1}^\infty  \left\langle (h(t)-\hat h(t))^*e_i,(h(t)-\hat h(t))^*e_i\right \rangle dt\\
&=\sum_{i=1}^\infty\int_0^\infty  \left\langle (h(t)-\hat h(t))^*e_i,(h(t)-\hat h(t))^*e_i\right \rangle dt.
\end{align*}
Where we used the fact that $h(t)-\hat h(t)$ is Hilbert-Schmidt to justify monotone convergence to pull out the summation. Next, since $(I-\sigma\pi)S_\infty=0$ we can manipulate the inner product inside the integral by observing that
\begin{align*}
h(t)-\hat h(t)&=C(I-\sigma \pi) S(t) B\\
&=C(I-\sigma \pi) (S(t) -S_\infty)B.
\end{align*}
Moving Hilbert adjoints inside the inner product, the expression before can thus be rewritten as
\begin{align*}
&\|\Sigma-\hat \Sigma\|^2_{\mathcal{H}_2}\\
=&\sum_{i=1}^\infty\int_0^\infty  \left\langle (C(I-\sigma \pi)P_\infty)^* e_i, (I-\sigma \pi)^*C^* e_i\right \rangle dt.
\end{align*}

Next, we want to move the integral inside of the inner product. To do this, we interpret the expression as an integral in the sense of Pettis and then in the next step use that this integral commutes with bounded operators. Therefore
\begin{align*}
\|\Sigma-\hat \Sigma\|^2_{\mathcal{H}_2}&=\tr\Bigg(C(I-\sigma \pi)P_\infty (I-\sigma \pi)^*C^*\Bigg).
\end{align*}
As before $P_\infty$ denotes the semistability Gramian defined in Definition \ref{semigram}. The operator Lyapunov equation for $P_\infty$ was shown to hold in Theorem \ref{semiLyapunov}. 
\end{proof}

Sufficient conditions for the impulse responses to be Hilbert-Schmidt arise when the input operator $B$ and the output operator $C$ are of finite rank, see \cite{curtain2001compactness}.

\subsection{Worked Example for the Heat Equation}
We now show how Theorem \ref{mytheorem} can be applied. Consider the example with the heated bar on $P=[0,1]$ with insulated boundary points $\pder{x}{p}(0,t)=\pder{x}{p}(1,t)=0$, initial distribution of heat $x(p,0)=x_0(p)$ and a source term $u$. This can be recast in terms of a system $A=\dder{p}$,  $\Sigma=(A, I, -)$ on  $L^2[0,1]$ with $D(A)=H^1_0$. Note that this model is not strictly stable but semistable since every function affine in $p$ is an element of $\ker A$ and thus for some initial conditional also a possible equilibrium distribution of heat.

Let $\pi$ be the projection on the first $N$ eigenvectors of $A$ and $\sigma$ the embedding of this into $L^2[0,1]$. It can be shown that the semigroup is given by (see \cite{curtain2012introduction})
 \begin{align*}
 S(t)x=\langle x, 1\rangle + \sum_{n=1}^\infty  e^{-n^2\pi^2 t}\langle x(\cdot), \cos(n \pi \cdot)  \rangle \cos(n \pi \cdot) .
 \end{align*}
 It is easy to see that for any $v = \pi x, x \in L^2[0,1]$ the reduced semigroup is given by
 \begin{align*}
 \hat S(t) v &=\langle v, 1\rangle + \sum_{n=1}^N  e^{-n^2\pi^2 t}\langle v(\cdot), \cos(n \pi \cdot)  \rangle \cos(n \pi \cdot)\\
 &=S(t)\pi x = \pi S(t) x,
 \end{align*}
which also verifies that $(\pi, \sigma, \hat S)$ makes for an invariant model reduction. Before we apply our Theorem, we need to verify that the Hilbert-Schmidt norm is finite
  \begin{align*}
 \tr(S^*S)&=\sum_{n=0}^\infty \langle S \cos(n\pi x), S\cos(n\pi x) \rangle \\&+\sum_{n=1}^\infty \langle S \sin(n\pi x), S\sin(n\pi x) \rangle\\
 &\leq 1 + \sum_{n=1}^\infty e^{-2n^2 \pi^2 t} <\infty
 \end{align*}
 for each $t\geq 0$, and so $S(t)$ is Hilbert-Schmidt so that Theorem \ref{mytheorem} is applicable.
 Now
\begin{align*}
&\tr((I-\sigma \pi)P_\infty (I-\sigma\pi))\\&=2\int_0^\infty  \sum_{n=N+1}^\infty\| e^{-n^2 \pi^2 t}\cos (n\pi \cdot)\|^2dt.
\end{align*}
Computing the integrand yields
\begin{align*}
\|e^{-n^2\pi^2 t} \cos (n \pi \cdot)\|^2 &=\int_0^1\left| e^{-n^2\pi^2 t} \cos (n\pi q) \right |^2 dq\\
&=e^{-2n^2\pi^2 t} \int_0^1 |\cos n \pi q |^2 dq\\
&=\frac{e^{-2n^2\pi^2 t}}{2}.
\end{align*}
We see via Theorem \ref{mytheorem} and integration that
\begin{align*}
\|\Sigma-\hat \Sigma\|_{\mathcal{H}_2}&=2\int_0^\infty  \sum_{n=N+1}^\infty\| e^{-n^2 \pi^2 t}\cos (n\pi \cdot)\|^2dt\\
&=2 \sum_{n=N+1}^\infty\int_0^\infty  \frac{e^{-2n^2\pi^2 t}}{2} dt  \\
&=\sum_{n=N+1}^\infty \frac{1}{\pi^2 n^2}
\end{align*}
where for instance monotone convergence can be used to exchange the sum and the integral.

\subsection{A Computational Perspective}
In the example above we relied heavily on the fact that the heat equation is such a well-studied object and that the semigroup $S(t)$ was explicitly available. In many applications, this is not always the case. Here, we sketch an alternative method for computation of the model error. We note that taken in combination with Theorem \ref{theexplicittheorem}, Theorem \ref{mytheorem}  allows us to compute the model error without explicit mention of the semigroup $S(t)$. Instead one may apply the following program:
\begin{enumerate}
\item Compute the kernel of $A$.
\item Find an arbitrary solution, $P$, of the semistability  Lyapunov equation.
\item Apply the projection onto the orthogonal complement of the kernel of $A$ according to $P_\infty = (I-S_\infty)P$.
\item Compute the trace as in Theorem \ref{mytheorem}.
\end{enumerate}

If $A$ is a partial differential operator this amounts to solving a sequence of partial differential equations, which in itself is often a difficult task. Nevertheless, when the semigroup $S(t)$ is particularly hard to compute, this provides an alternative path for the application of Theorem \ref{mytheorem}. One can imagine that this may also be useful for numerical evaluation of the error.

\section{Discussion and Conclusion}

Here we have investigated a class of model reductions based the commutativity assumption $\pi A = \hat A \pi$ and shown these to preserve semistability and approximate controllability. Further, we have given an exact a priori formula for the $\mathcal{H}_2$-errors for this class of model reductions and  further characterized the key quantity in this bound, known as the semistability Gramian, as the unique solution of an operator Lyapunov equation, satisfying an extra geometric constraint.

A natural focus for future research would be to do away with the commutativty assumption $\pi A = \hat A \pi$. This would be rather cumbersome as most proofs in Section 5 rest crucially on this assumption, but a reasonable attempt might involve first the computation of the $C_0$-semigroups of generators of the form $\pi A \sigma$ via for instance the Trotter product formula, \cite{trotter1959product} or by introduction of an error system.

We also note that Theorem \ref{theexplicittheorem} introduces a new unknown into the theory. For this theorem to achieve full potency, the computational feasibility of the semistability Lyapunov equation needs to be further investigated. Moreover, there is reason to believe that our results are applicable to numerical methods for partial differential equations such as finite element methods and it would be interesting to see if this connection could be of use. Here, we especially emphasize the shared use of both projective methods, as seen in the example, and the shared use of norms based on energy. 

\section*{Acknowledgements}
This work was supported in part by the Swedish Research Council (grant 2016-00861), and the Swedish Foundation for Strategic Research (project CLAS). The authors would also like to express their gratitude to Henrik Sandberg and Erik Lindgren for valuable discussions and insight. Finally, we thank three anonymous referees for their insightful comments.







\bibliographystyle{IEEEtranN}
\bibliography{confpaperv2}

\begin{thebibliography}{13}
\providecommand{\natexlab}[1]{#1}
\providecommand{\url}[1]{#1}
\csname url@samestyle\endcsname
\providecommand{\newblock}{\relax}
\providecommand{\bibinfo}[2]{#2}
\providecommand{\BIBentrySTDinterwordspacing}{\spaceskip=0pt\relax}
\providecommand{\BIBentryALTinterwordstretchfactor}{4}
\providecommand{\BIBentryALTinterwordspacing}{\spaceskip=\fontdimen2\font plus
\BIBentryALTinterwordstretchfactor\fontdimen3\font minus
  \fontdimen4\font\relax}
\providecommand{\BIBforeignlanguage}[2]{{%
\expandafter\ifx\csname l@#1\endcsname\relax
\typeout{** WARNING: IEEEtranN.bst: No hyphenation pattern has been}%
\typeout{** loaded for the language `#1'. Using the pattern for}%
\typeout{** the default language instead.}%
\else
\language=\csname l@#1\endcsname
\fi
#2}}
\providecommand{\BIBdecl}{\relax}
\BIBdecl

\bibitem[Hui(2011)]{hui2011optimal}
Q.~Hui, ``Optimal semistable control for continuous-time linear systems,''
  \emph{Systems \& Control Letters}, vol.~60, no.~4, pp. 278--284, 2011.

\bibitem[Besselink et~al.(2016)Besselink, Sandberg, and
  Johansson]{besselink2016clustering}
B.~Besselink, H.~Sandberg, and K.~H. Johansson, ``Clustering-based model
  reduction of networked passive systems,'' \emph{IEEE Transactions on
  Automatic Control}, vol.~61, no.~10, pp. 2958--2973, 2016.

\bibitem[Cheng et~al.(2017)Cheng, Kawano, and Scherpen]{cheng2017reduction}
X.~Cheng, Y.~Kawano, and J.~M. Scherpen, ``Reduction of second-order network
  systems with structure preservation,'' \emph{IEEE Transactions on Automatic
  Control}, vol.~62, no.~10, pp. 5026--5038, 2017.

\bibitem[Hui and Berg(2013)]{hui2013semistability}
Q.~Hui and J.~M. Berg, ``Semistability theory for spatially distributed
  systems,'' \emph{Systems \& Control Letters}, vol.~62, no.~10, pp. 862--870,
  2013.

\bibitem[Sasane(2002)]{sasane2002hankel}
A.~Sasane, \emph{Hankel norm approximation for infinite-dimensional
  systems}.\hskip 1em plus 0.5em minus 0.4em\relax Springer Science \& Business
  Media, 2002, vol. 277.

\bibitem[Atwell and King(2004)]{atwell2004reduced}
J.~A. Atwell and B.~B. King, ``Reduced order controllers for spatially
  distributed systems via proper orthogonal decomposition,'' \emph{SIAM Journal
  on Scientific Computing}, vol.~26, no.~1, pp. 128--151, 2004.

\bibitem[Curtain(2003)]{curtain2003model}
R.~F. Curtain, ``Model reduction for control design for distributed parameter
  systems,'' in \emph{Research directions in distributed parameter
  systems}.\hskip 1em plus 0.5em minus 0.4em\relax SIAM, 2003, pp. 95--121.

\bibitem[Diestel and Uhl(1977)]{diestel1977vector}
J.~Diestel and J.~J.~J. Uhl, \emph{Vector Measures}.\hskip 1em plus 0.5em minus
  0.4em\relax American Mathematical Society, Providence, RI, 1977.

\bibitem[Curtain and Zwart(2012)]{curtain2012introduction}
R.~F. Curtain and H.~Zwart, \emph{An introduction to infinite-dimensional
  linear systems theory}.\hskip 1em plus 0.5em minus 0.4em\relax Springer
  Science \& Business Media, 2012, vol.~21.

\bibitem[Bernstein and Bhat(1995)]{bernstein1995lyapunov}
D.~S. Bernstein and S.~P. Bhat, ``Lyapunov stability, semistability, and
  asymptotic stability of matrix second-order systems,'' \emph{Journal of
  Mechanical Design}, vol. 117, no.~B, pp. 145--153, 1995.

\bibitem[Atay and Roncoroni(2017)]{atay2017lumpability}
F.~M. Atay and L.~Roncoroni, ``Lumpability of linear evolution equations in
  banach spaces.'' \emph{Evolution Equations \& Control Theory}, vol.~6, no.~1,
  2017.

\bibitem[Curtain and Sasane(2001)]{curtain2001compactness}
R.~F. Curtain and A.~J. Sasane, ``Compactness and nuclearity of the hankel
  operator and internal stability of infinite-dimensional state linear
  systems,'' \emph{International Journal of Control}, vol.~74, no.~12, pp.
  1260--1270, 2001.

\bibitem[Trotter(1959)]{trotter1959product}
H.~F. Trotter, ``On the product of semi-groups of operators,''
  \emph{Proceedings of the American Mathematical Society}, vol.~10, no.~4, pp.
  545--551, 1959.

\end{thebibliography}

\end{document}